\documentclass[10pt,reqno,a4paper,oneside]{amsart}%
\usepackage{amsfonts}
\usepackage{amssymb}
\usepackage{amsmath}
\usepackage{amsthm}
\usepackage{epsfig}
\usepackage{float}
\usepackage{epstopdf}
\usepackage{graphicx}%
\providecommand{\U}[1]{\protect \rule{.1in}{.1in}}
\newtheorem{theorem}{Theorem}[section]

\newtheorem{lemma}[theorem]{Lemma}

\newtheorem{remark}{Remark}[section]
\newtheorem{problem}{Problem}
\newtheorem{example}{Example}[section]
\newtheorem{algorithm}[theorem]{Algorithm}

\theoremstyle{definition}
\theoremstyle{problem}
\theoremstyle{remark}
\numberwithin{equation}{section}

\ifx \pdfoutput \relax \let \pdfoutput=\undefined \fi
\newcount \msipdfoutput
\ifx \pdfoutput \undefined \else
\ifcase \pdfoutput \else
\ifx \paperwidth \undefined \else
\ifdim \paperheight=0pt\relax \else \pdfpageheight \paperheight \fi
\ifdim \paperwidth=0pt\relax \else \pdfpagewidth \paperwidth \fi
\fi \fi \fi

\begin{document}
%
\pagestyle{myheadings}%

\begin{center}
{\huge \textbf{Generalized low rank approximation to the symmetric positive semidefinite matrix}}\footnote{This research was
supported by the Natural Science Foundation of China (11601328), and the research fund of Shanghai Lixin University of Accounting and Finance (AW-22-2201-00118).
\par
{$^{*}$corresponding author}

\par
{$^{*}$ E-mail: hcychang@163.com
}}

\bigskip

{\large \textbf{Haixia Chang}$^{a,*}$}\bigskip \,\,\,\,    {\large \textbf{Chunmei Li}$^{b}$}\bigskip \,\,\,\,    {\large \textbf{Qionghui Huang}$^{b}$}

$^{a}$School of Statistics and Mathematics, Shanghai Lixin University of Accounting and Finance,\newline
Shanghai 201209, P.R. China\\

$^{b}$School of Mathematics and Computational Science, Guilin University of Electronic Technology, \\Guilin 541004, P.R. China\\

\end{center}

\bigskip

\begin{quotation}
\textbf{Abstract}
In this paper, we investigate the generalized low rank approximation to the symmetric positive semidefinite
matrix in the Frobenius norm:
$$\underset{ rank(X)\leq k}{\min} \sum^{m}_{i=1}\left \Vert
A_{i}-B_{i}XB_{i}^{T}\right \Vert^{2}_{F} ,$$
where $X$ is an unknown symmetric positive semidefinite matrix and $k$ is a positive integer.
We firstly use the property of a symmetric positive semidefinite matrix
$X=YY^{T}$, $Y$ with order $n\times k$, to convert the generalized low rank approximation into unconstraint generalized optimization problem.
Then we apply the nonlinear conjugate gradient method
 to solve the generalized optimization problem. We give a numerical example and an application in compressing and restoring of a color image.
\newline
\textbf{2000 AMS Subject Classifications\ }{\small 68W25, 65K10,15A33,
15A57}\newline
\end{quotation}

\section{\textbf{Introduction}}

Throughout this paper, let $\mathbb{R}$ and $\mathbb{R}^{m\times n}$ denote, respectively, the real number field  and the set of all real $m\times n$ matrices.
For a real or complex matrix $A$, the symbols $A^{T}$, $A^{*}$,
$r(A)$, $tr(A)$, $\Vert A \Vert_{F} $, $\Vert A \Vert_{2} $, $\Vert A \Vert$
stand for the transpose, the conjugate transpose, the rank, the trace, the Frobenius norm, the spectral norm, and any unitarily invariant norm of $A$, respectively.
 We write $A\geq 0$ if $A$ is a real symmetric positive semidefinite matrix.


In the last few years, the low rank matrix approximation and its generalizations have been one of the topics of very active research in matrix theory and the applications. The original low rank matrix approximation is due to the
 Eckart-Young theorem \cite{Young} in 1936, which was described as approximating one matrix by another
of lower rank is closely in distance and gave a constructive solution.
In 1960, Mirsky \cite{Mirsky} studied the problem
$$\Vert
A-\widehat{X} \Vert=\underset{ r(X)= k}{\min}\left \Vert
A-X\right \Vert, $$
and obtained the solution by the singular value decomposition(SVD).
In 1987, Golub, et al. \cite{Golub} gave a generalization of
the Eckart-Young-Mirsky matrix approximation theorem,
\begin{equation}\label{Golub}
\|(X_1,\widehat{X}_2)-(X_1,X_2)\|=\underset{ r(X_1,\overline{X}_2)\leq k}{\min}\|(X_1,\overline{X}_2)-(X_1,X_2)\|,
\end{equation}
which the columns of the initial matrix $X_1$ remains fixed.
 By (\ref{Golub}), the low rank approximation of initial matrix with some specified structure is called as the structured low rank matrix approximation, which can be written as
\begin{equation}\label{Su1}
\underset{ r(X)\leq k,X\in \Omega}{\min}\left\Vert
A-X\right\Vert,
\end{equation}
where $\Omega$ is a structure matrix set. For the complete survey of (\ref{Su1}), see \cite{CHU2003}, \cite{SLRA1}, \cite{SLRA2}. For the different matrix sets of $\Omega$ in (\ref{Su1}), such as the symmetric matrix \cite{7} , symmetric nonnegative definite matrix in \cite{Duan2014} \cite{2}, \cite{3}, \cite{12}, correlation matrix \cite{25} \cite{Duancorrelation}, Hankle matrix \cite{21}, circulate matrix \cite{13}, sylvester matrix in \cite{16}, \cite{23}, \cite{24}, and so on.

For the generalized forms of structured low rank approximation, there are some results, e.g.,
both X. Zhang etc \cite{Zhang2003} in 2003 and Wei etc \cite{Wei2007} in 2007 studied the fixed rank Hermitian nonnegative definite solution $X$ to the following fixed rank matrix approximation least squares problem
$$\underset{ r(X)= k}{\min} \left \Vert
A-BXB^{*}\right \Vert_{F}, $$
which discussed the ranges of the rank $k$ and derived expressions of the solutions by applying the SVD of the matrix of $B$.
In 2007, Friedland and Torokhti \cite{Friedland} considered
$$\underset{ r(X)\leq k}{\min}\left \Vert
A-BXC\right \Vert_{F}, $$
and applied SVD of matrices to give the explicit solution.
For the low rank approximation in the spectral norm,  the authors in \cite{Wei2012} and \cite{Shen2015} applied the norm-preserving dilation theorem and the matrix decomposition to obtain the explicit expression of the solution.


Motivated by the work mentioned above and keeping applications and interests
of low rank approximation in view, we in this paper consider
the generalized low rank approximation problem of the symmetric positive semidefinite matrix. The problem can be expressed as follows.

\begin{problem}
Given matrices $A_{i}\in R^{m_{i}\times m_{i}}$, $B_{i}\in R^{m_{i}\times n}$, $i=1,2,\ldots,m$, and an integer $k$,
 find an $n\times n$ real symmetric
positive semidefinite matrix $\widetilde{X}$ with $r(\widetilde{X})\leq k$ such that
 $$\sum^{m}_{i=1}\left \Vert
A_{i}-B_{i}\widetilde{X}B_{i}^{T}\right \Vert_{F}^{2}=\underset{X\geq 0, r(X)\leq k}{\min}\sum^{m}_{i=1}\left \Vert
A_{i}-B_{i}XB_{i}^{T}\right \Vert_{F}^{2} .$$
\end{problem}

The paper is organized as follows. We firstly use the property of a symmetric positive semidefinite matrix
$X=YY^{T}$, $Y\in R^{n\times k}$ to convert the generalized low rank approximation into unconstraint generalized optimization problem.
Then we apply the nonlinear conjugate gradient method with exact line search to solve the generalized optimization problem. Finally, we give a numerical example and an application to illustrate that the algorithm is effective.

\section{\textbf{Main results}}

In this section, we first characterize the feasible set, then transform Problem 1 into an unconstrained optimization problem. Finally we apply the nonlinear conjugate gradient method to solve it.

\begin{lemma}
(See \cite{FuzhenZhang}) An $n\times n$ matrix $X$ is real symmetric
positive semidefinite with $r(X)\leq k$ if and only if it can be written as $X=YY^{T}$, where $Y\in R^{n\times k}$.
\end{lemma}

The properties of the trace of a matrix can be referred arbitrary linear algebra book, e.g., \cite{FuzhenZhang}.

\begin{lemma} The matrices $A$, $B$, $C$, and $D$ with appropriate sizes, then
\begin{align}
tr(A^{T})&=tr(A),\nonumber\\
 tr(A+B)&=tr(A)+tr(B), \nonumber\\
 tr(AC)&=tr(CA), \nonumber\\
 tr(ACD)&=tr(CDA)=tr(DAC),\nonumber\\
 \Vert A\Vert_{F}^{2}&=tr(A^{T}A).\nonumber
\end{align}
\end{lemma}

\begin{lemma}(See \cite{Cookbook})
 Let $A,B,C$ be the constant matrices with appropriate size, and $X$, $Y$ be variable matrices. Then there are the following rules about deriving the differential of the matrix expressions:
\begin{align}
\partial A& =0,  \nonumber \\
\partial(X+Y)& =\partial(X)+\partial(Y),\nonumber\\
\partial(tr(X))& =tr(\partial(X)),\nonumber\\
\partial(XY)& =\partial(X)Y+X\partial(Y),\nonumber\\
\frac{\partial}{\partial X}tr(X^{T}BX)& =BX+B^{T}X,\nonumber\\
\frac{\partial}{\partial X}tr(XBX^{T})& =XB^{T}+XB,\nonumber\\
\frac{\partial}{\partial X}tr(AXBX)& =A^{T}X^{T}B^{T}+B^{T}X^{T}A^{T},\nonumber\\
\frac{\partial}{\partial X}tr(B^{T}X^{T}CXX^{T}CXB)& = CXX^{T}CXBB^{T}+C^{T}XBB^{T}X^{T}C^{T}X \nonumber\\
&\,\,\,\,\,\,\,+CXBB^{T}X^{T}CX+C^{T}XX^{T}C^{T}XBB^{T}.\nonumber
\end{align}

\end{lemma}

By Lemma 2.1, we write $X=YY^{T}$ in Problem 1, where $Y\in R^{n\times k}$. The Problem 1 can be as follows.

\begin{problem}
Given matrices $A_{i}\in R^{m_{i}\times m_{i}}$, $B_{i}\in R^{m_{i}\times n}$, $i=1,2,\ldots,m$, and an integer $k$,
 find a matrix $\widetilde{Y}\in R^{n\times k}$ such that
 $$\sum^{m}_{i=1}\left \Vert
A_{i}-B_{i}\widetilde{Y}\widetilde{Y}^{T}B_{i}^{T}\right \Vert^{2}=\underset{Y\in R^{n\times k}}{\min}\sum^{m}_{i=1}\left \Vert
A_{i}-B_{i}YY^{T}B_{i}^{T}\right \Vert^{2} .$$
\end{problem}

We find Problem 2 is an unconstrained nonlinear matrix optimization.
Assume the function
\begin{equation}
f(Y)=\sum^{m}_{i=1}\left \Vert
A_{i}-B_{i}YY^{T}B_{i}^{T}\right \Vert_{F}^{2}, \,\,\,\,\,\,Y\in R^{n\times k}, \label{fy}%
\end{equation}
which defines a map $R^{n\times k}\rightarrow R$. It is easy to verify that Problem 2 is equivalent to the following problem
\begin{equation}
\underset{Y\in R^{n\times k}}{\min}f(Y).   \label{mfy}
\end{equation}

\begin{theorem}
The gradient of the objective function $f(Y)$ in (\ref{fy}) is
\begin{equation}
\bigtriangledown f(Y)=\sum^{m}_{i=1}(4 B_{i}^{T} B_{i}YY^{T} B_{i}^{T}B_{i}Y-2B_{i}^{T}A_{i}B_{i}Y-2B_{i}^{T}A_{i}^{T}B_{i}Y)   \label{grad}
\end{equation}
\end{theorem}

\begin{proof}
According to Lemma 2.2,
\begin{align}
\left \Vert
A_{i}-B_{i}YY^{T}B_{i}^{T}\right \Vert_{F}^{2}
&  =tr[(A_{i}-B_{i}YY^{T}B_{i}^{T})^{T}(A_{i}-B_{i}YY^{T}B_{i}^{T})] \nonumber \\
&  =tr(B_{i}YY^{T}B_{i}^{T}B_{i}YY^{T}B_{i}^{T})-tr(B_{i}YY^{T}B_{i}^{T}A_{i})-tr(A_{i}^{T}B_{i}YY^{T}B_{i}^{T})+tr(A_{i}^{T}A_{i})  \label{s10}%
\end{align}
We get by (\ref{s10})
\begin{align}
 f(Y) &  =\sum^{m}_{i=1}\left \Vert
A_{i}-B_{i}YY^{T}B_{i}^{T}\right \Vert^{2} \nonumber \\
&  =\sum^{m}_{i=1}[tr(B_{i}YY^{T}B_{i}^{T}B_{i}YY^{T}B_{i}^{T})-tr(B_{i}YY^{T}B_{i}^{T}A_{i})-tr(A_{i}^{T}B_{i}YY^{T}B_{i}^{T})+tr(A_{i}^{T}A_{i}) ]  \label{s1}%
\end{align}
By Lemma 2.3 and the equality (\ref{s1}), the gradient of $f(Y)$ can be expressed as
\begin{align}
\bigtriangledown f(Y)&=\frac{\partial}{\partial Y} f(Y)\nonumber \\
& =\sum^{m}_{i=1}[ \frac{\partial}{\partial Y} tr(B_{i}YY^{T}B_{i}^{T}B_{i}YY^{T}B_{i}^{T})-\frac{\partial}{\partial Y} tr(B_{i}YY^{T}B_{i}^{T}A_{i})-\frac{\partial}{\partial Y} tr(A_{i}^{T}B_{i}YY^{T}B_{i}^{T})+\frac{\partial}{\partial Y} tr(A_{i}^{T}A_{i}) ]\label{s1e}
\end{align}
By Lemma 2.2, note that
$$
tr(B_{i}YY^{T}B_{i}^{T}B_{i}YY^{T}B_{i}^{T})=tr(B_{i}^{T}B_{i}YY^{T}B_{i}^{T}B_{i}Y Y^{T})=tr(Y^{T}B_{i}^{T}B_{i}YY^{T}B_{i}^{T}B_{i}Y)
$$
$$tr(B_{i}YY^{T}B_{i}^{T}A_{i})=tr(B_{i}^{T}A_{i}B_{i}YY^{T})=tr(Y^{T}B_{i}^{T}A_{i}B_{i}Y)$$
$$tr(A_{i}^{T}B_{i}YY^{T}B_{i}^{T})=tr(B_{i}^{T}A_{i}^{T}B_{i}YY^{T})=tr(Y^{T}B_{i}^{T}A_{i}^{T}B_{i}Y)$$
Applying Lemma 2.3, we obtain the following equalities
\begin{align}
\frac{\partial}{\partial Y} tr(B_{i}YY^{T}B_{i}^{T}B_{i}YY^{T}B_{i}^{T}) & =\frac{\partial}{\partial Y}tr(Y^{T}B_{i}^{T}B_{i}YY^{T}B_{i}^{T}B_{i}Y) \nonumber \\
& =4 B_{i}^{T} B_{i}YY^{T} B_{i}^{T}B_{i}Y   \label{s1a}
\end{align}
\begin{align}
\frac{\partial}{\partial Y} tr(B_{i}YY^{T}B_{i}^{T}A_{i})& =\frac{\partial}{\partial Y} tr(Y^{T}B_{i}^{T}A_{i}B_{i}Y)\nonumber \\
&=B_{i}^{T}A_{i}B_{i}Y+ B_{i}^{T}A_{i}^{T}B_{i}Y   \label{s1b}
\end{align}
\begin{align}
\frac{\partial}{\partial Y} tr(A_{i}^{T}B_{i}YY^{T}B_{i}^{T})& =\frac{\partial}{\partial Y}tr(Y^{T}B_{i}^{T}A_{i}^{T}B_{i}Y)\nonumber \\
& =B_{i}^{T}A_{i}^{T}B_{i}Y+ B_{i}^{T}A_{i}B_{i}Y    \label{s1c}
\end{align}
\begin{equation}
\frac{\partial}{\partial Y} tr(A_{i}^{T}A_{i})=0  \label{s1d}
\end{equation}
Substituting (\ref{s1a}), (\ref{s1b}), (\ref{s1c}), (\ref{s1d}) into (\ref{s1e}), we verify the equality (\ref{grad}) holds.
\end{proof}

In Theorem 2.4, we obtain the gradient of $f(Y)$. In order to solve the minimization problem (\ref{mfy}), we apply the nonlinear conjugate gradient method with exact line search. For the nonlinear conjugate gradient method, it can be refered to \cite{nc1},\cite{nc2}. The following is the algorithm of solving (\ref{mfy}).
\begin{algorithm}
$1$. Given matrices $A_{i}\in R^{m_{i}\times m_{i}}$, $B_{i}\in R^{m_{i}\times n}$, $i=1,2,\ldots,m$, initial matrix $Y_{0}\in R^{n\times k}$, and tolerant error $\varepsilon >0$;\\
$2$. Evaluate $f_{0}=f(Y_{0})$,$\bigtriangledown f_{0}=\bigtriangledown f(Y_0)$, $D_{0}=-\bigtriangledown f(Y_0)$, $k=0$;\\
$3$. When $\Vert\bigtriangledown f_{k}\Vert_{F} > \varepsilon $,
find $t_{k}$ such that \\
$$ \sum^{m}_{i=1}\left \Vert
A_{i}-B_{i}(Y_k+t_kD_k)(Y_k+t_kD_k)^{T}B_{i}^{T}\right \Vert_{F}^{2}=\underset{t>0 }{\min} \sum^{m}_{i=1}\left \Vert
A_{i}-B_{i}(Y_k+tD_k)(Y_k+tD_k)^{T}B_{i}^{T}\right \Vert_{F}^{2},$$
$Y_{k+1}=Y_{k}+t_{k}D_{k},$\\
$\bigtriangledown f_{k+1}=\bigtriangledown f(Y_{k+1}),$\\
$\beta_{k+1}=\frac{\bigtriangledown f_{k+1}^{2}}{\bigtriangledown f_{k}^{2}},$\\
$D_{k+1}=-\bigtriangledown f_{k+1}+\beta_{k+1}D_{k}$\\
end
\end{algorithm}

\begin{remark}
  Algorithm 2.5 is iplemented with exact line search for a step length $t_k$. We can use the exact line search method in \cite{line}, \cite{Duan2014} to compute the step length $t_k$, because the univariate function $\phi(\cdot)$ defined by
  $$\phi(t)=\sum^{m}_{i=1}\left \Vert
A_{i}-B_{i}(Y_k+tD_k)(Y_k+tD_k)^{T}B_{i}^{T}\right \Vert_{F}^{2}=a_{4}t^{4}+a_{3}t^{3}+a_{2}t^{2}+a_{1}t+a_{0}$$
where
\begin{align}
&a_{4}=\sum^{m}_{i=1}\left \Vert B_{i}D_{k}D_{k}^{T}B_{i}^{T}\right \Vert_{F}^{2},\nonumber \\
&a_{3}=\sum^{m}_{i=1} 2tr[B_{i}D_{k}D_{k}^{T}B_{i}^{T}ㄗY_{k}D_{k}^{T}+D_{k}Y_{k}^{T}ㄘ]\nonumber \\
&a_{2}=\sum^{m}_{i=1}-2tr[B_{i}D_{k}D_{k}^{T}B_{i}^{T}( A_{i}-B_{i}Y_{k}Y_{k}^{T}B_{i}^{T})]\nonumber \\
&a_{1}=\sum^{m}_{i=1}-2tr[(A_{i}-B_{i}YY_{k}^{T}B_{i}^{T})(Y_{k}D_{k}^{T}+D_{k}Y_{k}^{T})]\nonumber \\
&a_{0}=\sum^{m}_{i=1}\left \Vert A_{i}-B_{i}Y_{k}Y_{k}^{T}B_{i}^{T}\right \Vert_{F}^{2},\nonumber
\end{align}
is quadratic, which is similar as the metric function for Newton's method with line search for solving Algebraic Riccati equation in \cite{line}.
\end{remark}

\begin{remark} Note that
 $D_{k+1}$ in Algorithm 2.5 is a descent direction.  In fact, we know that the exact line search always satisfies the following equality
  $$tr[(\nabla f_{k+1})^{T}D_{k}]=[vec(f_{k+1})^{T}D_{k})^{T}vec(D){k})]=0$$
By using $vec(\cdot)$ to the equality $D_{k+1}=-\bigtriangledown f_{k+1}+\beta_{k+1}D_{k}$ and premultiplying by $[vec(\bigtriangledown f_{k+1})]^{T}$, we get
$$[vec(\bigtriangledown f_{k+1})]^{T}vec(D_{k+1})=-\|\bigtriangledown f_{k+1})\|_{F}^{2}+\beta_{k+1}vec(\bigtriangledown f_{k+1})]^{T}vec(D_{k}).$$
Hence, we obtain $[vec(\bigtriangledown f_{k+1})]^{T}vec(D_{k+1})<0,$ which implies that $D_{k+1}$ is a descent direction.

\end{remark}

\begin{theorem}
Suppose that $f$ is continuously differentiable and bounded below. If the gradient $\bigtriangledown f$ is lipschitz continuous, that is , there exists a
constant $L$ such that
$$\left \Vert
\bigtriangledown f(X)-\bigtriangledown f(Y)\right \Vert\leq L \left \Vert
 X- Y\right \Vert , \,\,\,\,\,\forall X,Y\in R^{n\times k},$$
Then the sequence ${Y_{k}}$ generated by Algorithm 2.1 satisfies
$$\lim_{n\rightarrow \infty} \inf\Vert\bigtriangledown f(Y_{k}) \Vert=0.$$
\end{theorem}

\section{\textbf{A numerical example}}

In this section, we use some numerical examples to illustrate the Algorithm 2.5 is feasible and effective to solve Problem 1. All tests are performed by using $Matlab$ $R2016a$. We denote the relative residual error
\begin{align}
\varepsilon(k)=\frac{\sum\nolimits^m_{i=1}\|A_{i}-B_{i}Y_{k}Y_{k}^{T}B_{i}^{T}\|_{F}}{\sum\nolimits^m_{i=1}\|A_{i}\|_F}.\nonumber
\end{align}
\begin{flushleft}
and the gradient norm
\end{flushleft}
\begin{center}$\|g_{k}\|_F=\|\nabla f(Y_{k})\|_F$,\end{center}
where $Y_{k}$ is the $t$th iterative matrix of Algorithm. We use the stopping criterion
\begin{center}$\|g_{k}\|_F<1.0\times10^{-4}$.\end{center}
And the initial value $Y_{0}$ is randomly generated by the \emph{rand} function in MATLAB.

\begin{example}
Consider problem 1 with $m=2$ and
\begin{align*}
A_{1}  &  =\left[
\begin{array}
[c]{cccc}%
0.6938 & 0.1093 & 0.0503 & 0.8637\\
0.9452 & 0.3899 & 0.2287 & 0.0781\\
0.7842 & 0.5909 & 0.8342 & 0.6690\\
0.7056 & 0.4594 & 0.0156 & 0.5002
\end{array}
\right]  \text{, }\\
A_{2}  &  =\left[
\begin{array}
[c]{ccccc}%
0.2180 & 0.5996 & 0.0196 & 0.5201\\
0.5716 & 0.0560 & 0.4352 & 0.8639 \\
0.1222 & 0.0563 & 0.8322 & 0.0977 \\
0.6712 & 0.1523 & 0.6174 & 0.9081
\end{array}
\right]  \text{, }\\
B_{1}  &  =\left[
\begin{array}
[c]{ccccc}%
0.1080 & 0.0046 & 0.9870 & 0.5078\\
0.5170 & 0.7667 & 0.5051 & 0.5856 \\
0.1432 & 0.8487 & 0.2714 & 0.7629 \\
0.5594 & 0.9168 & 0.1008 & 0.0830
\end{array}
\right]  \text{, }\\
B_{2}  &  =\left[
\begin{array}
[c]{ccccc}%
0.6616 & 0.5905 & 0.4519 & 0.6801\\
0.5170 & 0.4406 & 0.8397 & 0.3672 \\
0.1710 & 0.9419 & 0.5326 & 0.2393 \\
0.9386 & 0.6559 & 0.5539 & 0.5789
\end{array}
\right]  \text{. }
\end{align*}
Case I: Set $k=2$. We use Algorithm 2.5 with the initial value
\begin{align*}
Y_{0}  &  =\left[
\begin{array}
[c]{cccc}%
0.8669 & 0.3002 \\
0.4068 & 0.4014 \\
0.1126 & 0.8334 \\
0.4438 & 0.4036
\end{array}
\right]
\end{align*}
to solve problem (2.2). And we get the solution $\hat{Y}$ of problem (2.2)
\begin{align*}
\hat{Y}\approx Y_{242}  &  =\left[
\begin{array}
[c]{cccc}%
0.7015 & -1.0397 \\
-0.7793 & 0.3971 \\
-0.5158 & -0.2175 \\
-0.1875 & -0.0296
\end{array}
\right] \text{. }
\end{align*}
Hence, the solution $\hat{X}$ of problem 1 is
\begin{align*}
\hat{X}=\hat{Y}\hat{Y}^T  &  =\left[
\begin{array}
[c]{cccc}%
1.5731 & -0.9596 & -0.1357 & 0.1008 \\
-0.9596 & 0.7651 & 0.3156 & 0.1344 \\
-0.1357 & 0.3156 & 0.3133 & 0.1032 \\
-0.1008 & 0.1344 & 0.1032 & 0.0360
\end{array}
\right] \text{. }
\end{align*}
And the curves of the relative residual error $\varepsilon(k)$ and the gradient norm $\|\nabla f(Y_{k})\|_{F}$ are in Fig.1.
\begin{figure}
\centering
\includegraphics[height=5cm,width=13cm]{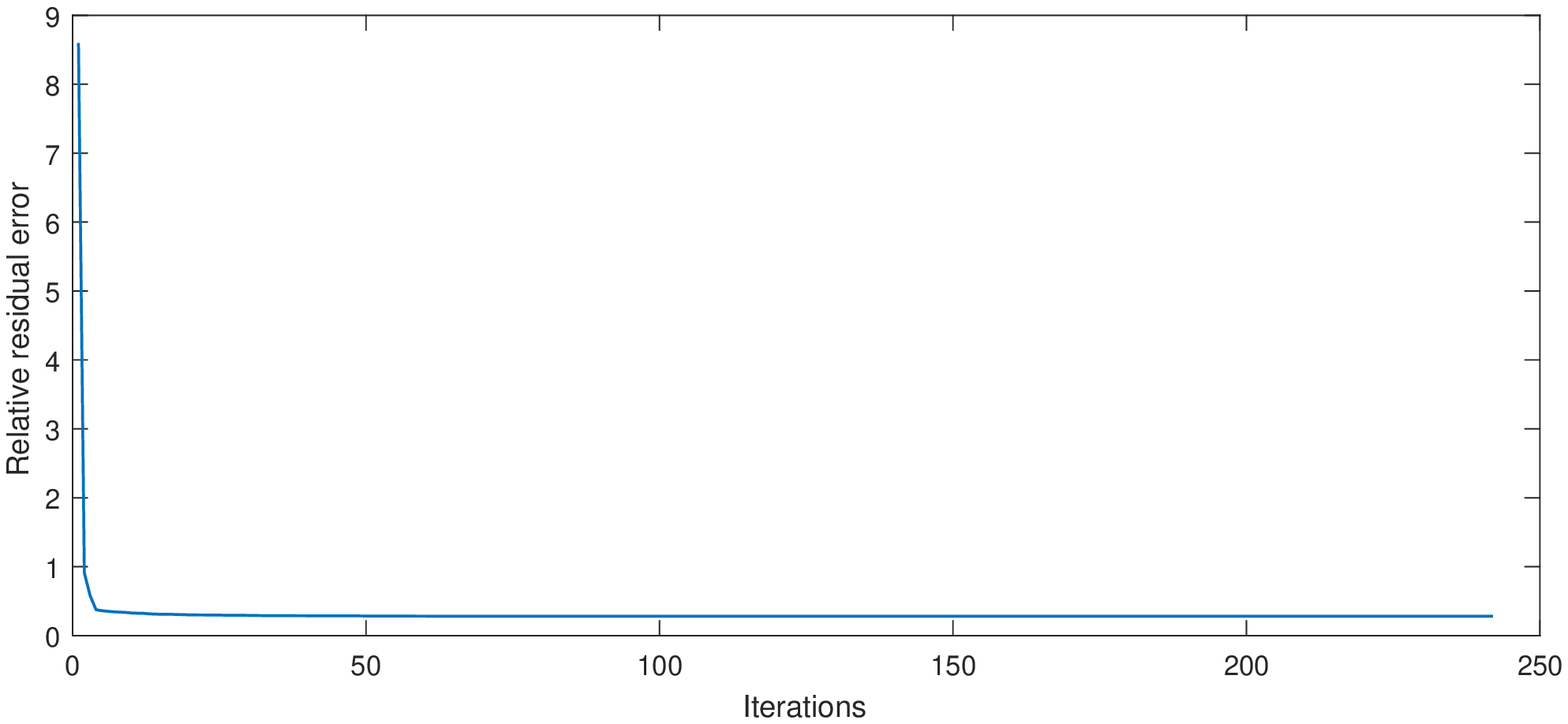}
\end{figure}
\begin{figure}
\centering
\includegraphics[height=5cm,width=13cm]{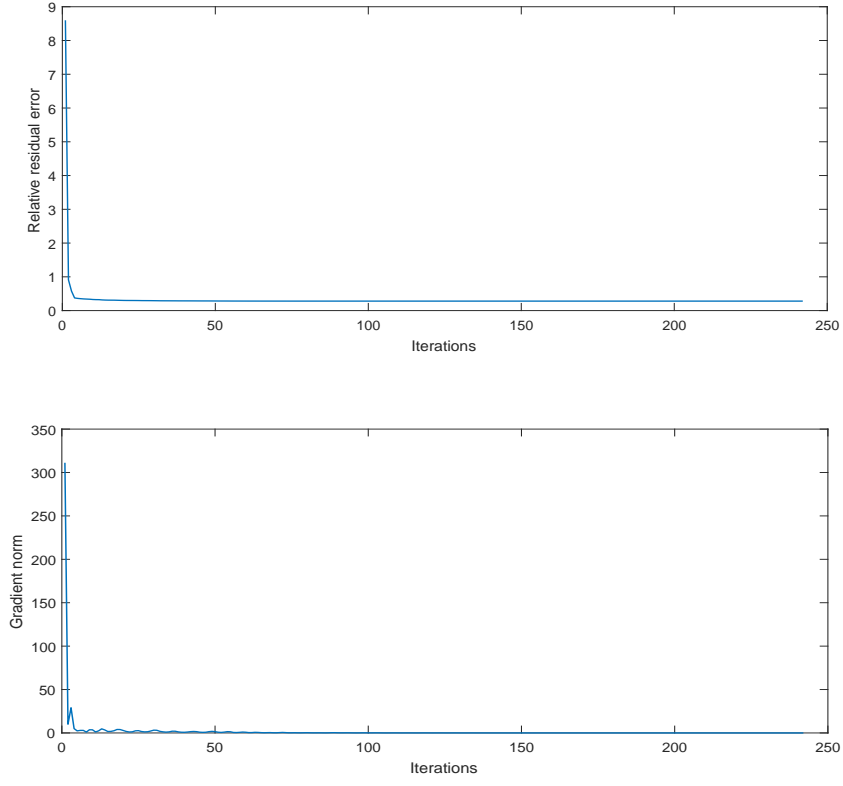}
\caption{Convergence curves of the relative residual error $\varepsilon(k)$ and the gradient norm $\|\nabla f(Y_{k})\|_{F}$}
\label{1}
\end{figure}

Case II: Set $k=3$. We use Algorithm 2.5 with the initial value
\begin{align*}
Y_{0}  &  =\left[
\begin{array}
[c]{cccc}%
0.5211 & 0.6791 & 0.0377 \\
0.2316 & 0.3955 & 0.8852 \\
0.4889 & 0.3674 & 0.9133 \\
0.6241 & 0.9880 & 0.7962
\end{array}
\right]\text{. }
\end{align*}
to solve problem (2.2). And we get the solution $\hat{Y}$ of problem(2.2)
\begin{align*}
\hat{Y}\approx Y_{224} &  =\left[
\begin{array}
[c]{cccc}%
0.4617 & 0.4371 & -1.0811 \\
-0.4498 & -0.5971 & 0.4541 \\
-0.2575 & -0.4660 & -0.1729 \\
-0.0979 & -0.1620 & -0.0141
\end{array}
\right]\text{. }
\end{align*}
Hence, the solution $\hat{X}$ of problem 1 is
\begin{align*}
\hat{X}=\hat{Y}\hat{Y}^T  &  =\left[
\begin{array}
[c]{cccc}%
1.5731 & -0.9596 & -0.1357 & 0.1008 \\
-0.9596 & 0.7650 & 0.3156 & 0.1344 \\
-0.1357 & 0.3156 & 0.3133 & 0.1032 \\
-0.1008 & 0.1344 & 0.1032 & 0.0360
\end{array}
\right]
\end{align*}
And the curves of the relative residual error $\varepsilon(k)$ and the gradient norm $\|\nabla f(Y_{k})\|_{F}$ are in Fig.2.
\end{example}

Example 3.1 shows that Algorithm 2.5 is feasible to solve problem 1.
\begin{figure}
\centering
\includegraphics[height=5cm,width=13cm]{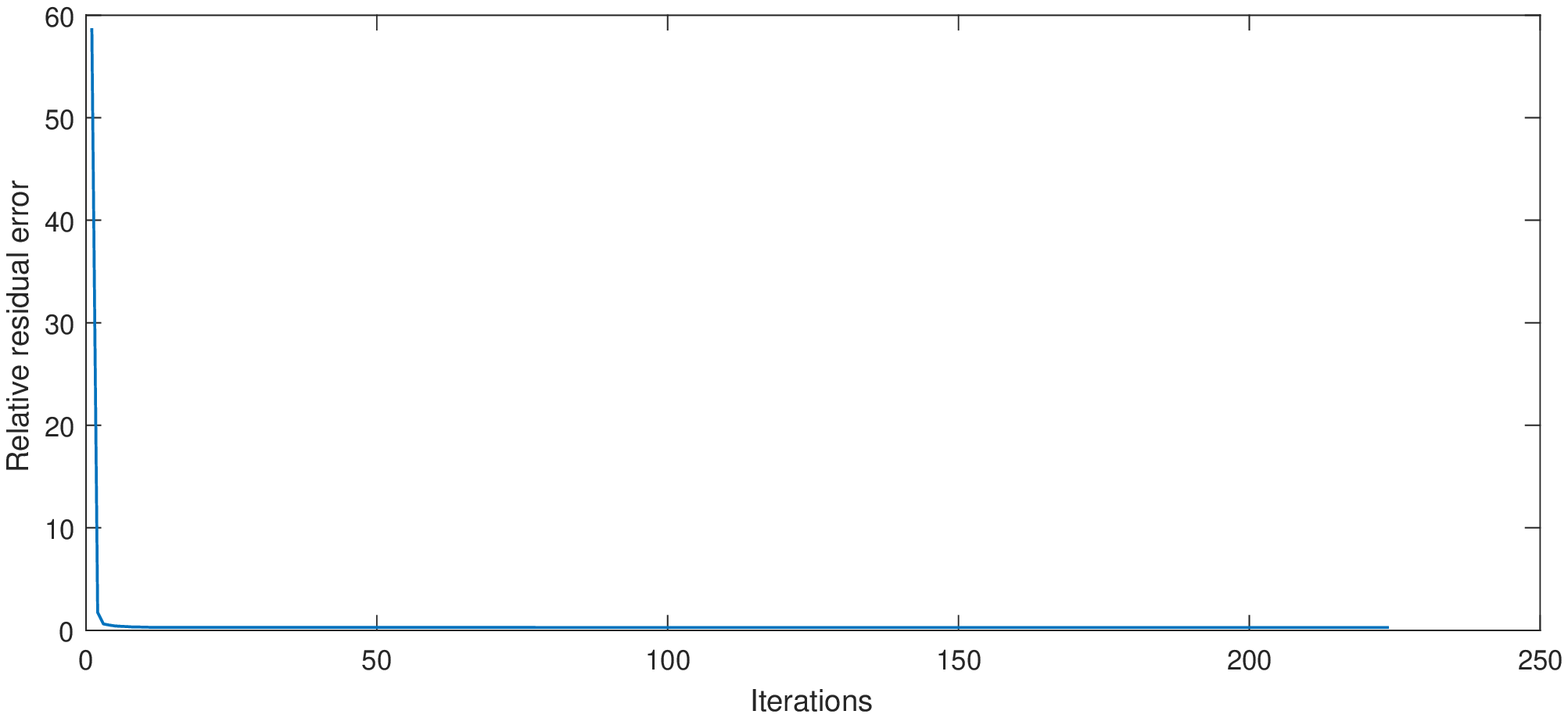}
\end{figure}
\begin{figure}
\centering
\includegraphics[height=5cm,width=13cm]{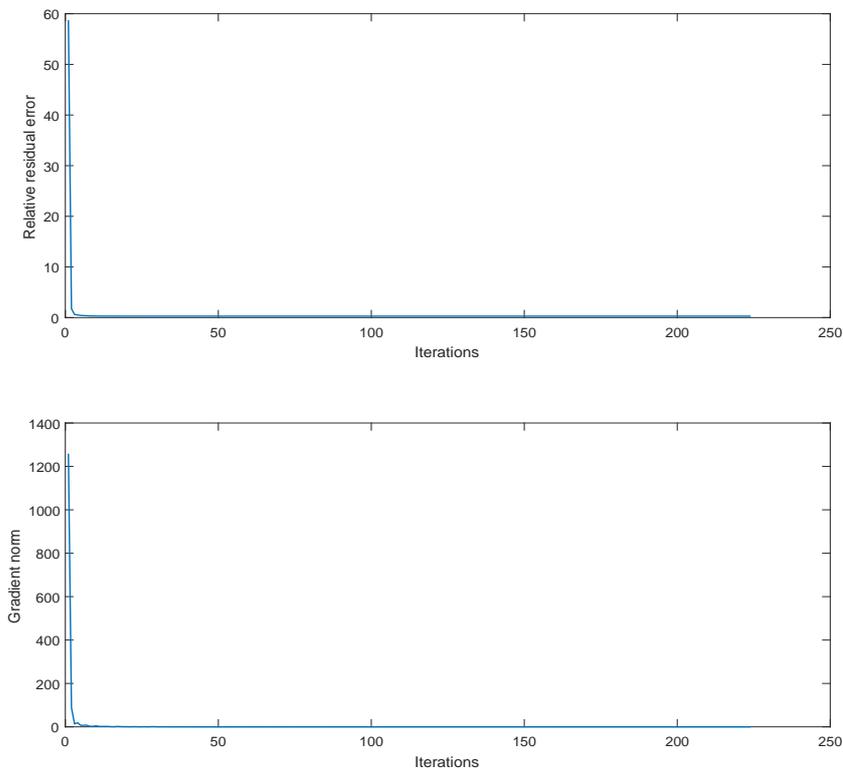}
\caption{Convergence curves of the relative residual error $\varepsilon(t)$ and the gradient norm $\|\nabla f(Y_{t})\|_{F}$}
\label{1}
\end{figure}

\begin{example}
There are three observed positive semidefinited images $A_{1},A_{2},A_{3}$, which are generated from image $A$ under three different transformations. By using Algorithm 2.5, we find the approximation matrix $X$ such that $\underset{ rank(X)\leq k}{\min} \sum^{m}_{i=1}\left \Vert
A_{i}-P_{i}XP_{i}^{T}\right \Vert^{2}_{F}$ ,
\end{example}

\section{\textbf{Conclusion}}

We in this paper have solved the generalized low rank approximation of a symmetric positive semidefinite matrix.  We convert the generalized low rank approximation into unconstraint generalized optimization problem.
We apply the nonlinear conjugate gradient method with exact line search to solve the generalized optimization problem. The numerical examples show that the algorithm is feasible.


\begin{thebibliography}{99}                                                                                               %

\bibitem {FuzhenZhang}Fuzhen Zhang, Matrix Theory: Basic results and Techniques, Springer-Verlag, 2nd edition, 2011.

\bibitem {Duan2014}X.F. Duan, J.F.Li, Q.W. Wang, X.J. Zhang, Low rank approximation
of the symmetric positive semidefinite matrix, J. Comput. Appl. Math. 260(2014) 236-243.

\bibitem {LB}L.Eld\'{e}n, B.Savas, Permutation theory and optimality conditions for the best multilinear rank approximation of a tensor, SIAM J. Matrix Anal. Appl., 32(2011) 1422-1450.



\bibitem {Cookbook} K. B. Petersen and M. S. Pedersen, The matrix cookbook, Technical report, version: November 15, 2012,  http://matrixcookbook.com.


 \bibitem {Young} C. Eckart and G. Young, The approximation of one matrix by another of lower
rank, Psychometrika 1:211-218 (1936).

\bibitem {Golub} G. H. Golub, A. Hoffman, and G. W. Stewart, A generalization of the Eckart-
Young-Mirsky matrix approximation theorm, Linear Algebra Appl. 88/89:317-327
(1987).


\bibitem {Mirsky}L. Mirsky, Symmetric gauge functions and unitarily invariant norms, @art. J.
Math. word 11:50-59 (1960).





\bibitem {Friedland} S. Friedland and A. Torokhti, Generalized rank-constrained matrix approximations, SIAM
J. Matrix Anal. Appl., 29(2007) 656每659.

\bibitem {Macomp} G. H. Golub, C.F. Van Loan, Matrix Computations, 3rd ed., Baltimore, 1996.

\bibitem {lra2011} N. Gillis and F. Glineur, Low-rank matrix approximation with weights or missing data is
NP-hard, SIAM J. Matrix Anal. Appl., 32(2011) 1149-1165.

\bibitem {Horn} R. Horn, C. Johnson, Matrix Analysis, Cambridge University Press, Cambridge, UK,
1990.

\bibitem {Wang2017} X.Liu, W.Li, H.Wang, Rank constrained matrix best approximation problem with respect to (skew) Hermitian matrices, J. Comput. Appl. Math., 319(2017) 77-86.
101每107.

\bibitem {CHU2003} M.T. Chu, R.E. Funderlic, R.J. Plemmons, Structured low rank approximation, Linear Algebra Appl. 366 (2003) 157每172.


\bibitem{2} F.R. Bach, M.I. Jordan, Kernel independent component analysis, J. Mach. Learn. Res. 3 (2002) 1每48.

\bibitem{3} L. Boman, H. Koch, A.S. de Meras, Method specific Cholesky decomposition: coulomb and exchange energies, J. Chem. Phys. 129 (2008) 107每134.

\bibitem {7} M.T. Chu, R.J. Plemmons, Real-value, low rank, circulant approximation, SIAM J. Matrix Anal. Appl. 24 (2003) 645每659.

\bibitem {12} H. Harbrecht, M. Peters, R. Schneider, On the low rank approximation by the pivoted Cholesky decomposition, Appl. Numer. Math. 62 (2012) 428每440.

\bibitem {13} U. Helmke, M.A. Shayman, Critical points of matrix least squares distance function, Linear Algebra Appl. 215 (1995) 1每19.
\bibitem {16} B.Y. Li, Z. Yang, L.H. Zhi, Fast low rank approximation of a Sylvester matrix by structured total least norm, Japan Soc. Symbolic Algebr. Comput. 11
(2005) 165每174.
\bibitem {21} H. Park, L. Zhang, J.B. Rosen, Low rank approximation of a Hankel matrix by structured total norm, BIT 39 (1999) 757每779.
\bibitem {23} J. Winkler, J. Allan, Structured low rank approximations of the Sylvester resultant matrix for approximation gcds of bernstein basis polynomials, Electron. Trans. Numer. Anal. 31 (2008) 141每155.

\bibitem {24} J. Winkler, J. Allan, Structured total least norm and approximate gcds of inexact polynomials, J. Comput. Appl. Math. 215 (2008) 1每13.

\bibitem {25} Z.Y. Zhang, Optimal low rank approximation to a correlation matrix, Linear Algebra Appl. 364 (2003) 161每187.

\bibitem {Duancorrelation} X.F. Duan, J.C.Bai, M.J. Zhang, X.J. Zhang, On the generalized low rank approximation
of the correlation matrices arising in the asset portfolio, Linear Algebra Appl. 461 (2014) 1每17.

\bibitem {LIN2011} M.M. Lin, Discrete Eckart-Young theorem for integer matrices. SIAM J. Matrix Anal. Appl., 32(2011), 1367-1382.

\bibitem {br}B. Recht, M.Fazel, P.A. Parrilo, Guaranteed minimum-rank solutions of linear matrix equations via nuclear norm minimization, SIAM Rev., 52(2010) 471-501


\bibitem {SLRA1}I. Markovsky, Structured low rank approximation and its applications, Automatica 44 (2008) 891每909.

\bibitem {SLRA2} I. Markovsky, Low Rank Approximation: Algorithms, Implementation, Application, Springer-Verlag, Berlin, 2012.




\bibitem {Wei2012} M.Wei, D.Shen, Minimum rank solution to the matrix approximation problems in the spectral norm,
Siam J. Matrix Anal. Appl., 33(3)(2012) 940-957.


\bibitem {Shen2015} D.Shen, M.Wei Y.Liu, Minimum rank (skew)Hermitian solutions to the matrix approximation problem in the spectral norm,
J. Comput. Appl. Math., 288(2015) 351-365.

\bibitem {Zhang2003} X. Zhang, M.Y. Zhang, The rank-constrained Hermitian nonnegative-definite and positive -definite solutions to the matrix
 $AXA^{*}=B$, Linear Algebra Appl., 370(2003) 163-174.

\bibitem {Wei2007} M.Wei, Q. Wang, On rank-constrained Hermitian nonnegative-definite least squares solutions to the matrix equation $AXA^{H}=B,$
Int. J. Comput. Math., 84(2007) 945-952.

\bibitem {nc1} D.H. Li, X.J. Tong, Z. Wan, Numerical Optimization, Science Press, Beijing, 2005.

\bibitem {nc2} J. Nocedal, S.J. Wright, Numerical Optimization, Springer-Verlag, Berlin, 1999.

\bibitem {line} P. Benner, R. Byers, A exact line searches method for generalized continuous-time algebra Riccati equations,
IEEE Trans. Automat. Control, 43(1998), 101-107.





\end{thebibliography}
\end{document}